\documentclass[11pt]{amsart}
\usepackage{amsthm}
\usepackage{amsmath,amsfonts,amssymb,epsfig,amsbsy, enumerate}
\usepackage{caption, subfloat}
\usepackage{fancyhdr}

\def\ML{\operatorname{ML}}

\def\MCG{\operatorname{MCG}}
\def\Thu{\operatorname{Thu}}

\newtheorem{thm}{Theorem}[section]
\newtheorem{lemma}[thm]{Lemma}
\newtheorem{prop}[thm]{Proposition}
\newtheorem{remark}[thm]{Remark}
\newtheorem{cor}[thm]{Corollary}

\theoremstyle{definition}

\numberwithin{equation}{section}

\begin{document}

\title[]{A remark on the word length in surface groups}
\author{Viveka Erlandsson}
\address{University of Bristol}
\email{v.erlandsson@bristol.ac.uk}
\thanks{\today}

\begin{abstract}
Let $\Sigma$ be a surface of negative Euler characteristic and $S$ a generating set for $\pi_1(\Sigma,p)$  consisting of simple loops that are pairwise disjoint (except at $p$). We show that the word length with respect to $S$ of an element of $\pi_1(\Sigma,p)$ is given by its intersection number with a well-chosen collection of curves and arcs on $\Sigma$. The same holds for the word length of (a free homotopy class of) an immersed curve on $\Sigma$. As a consequence, we obtain the asymptotic growth of the number of immersed curves of bounded word length, as the length grows, in each mapping class group orbit.
\end{abstract}

\maketitle

\section{Introduction}

Let $\Sigma$ be a compact, connected, orientable surface of genus $g$ and with $r$ (possibly none) boundary components. To rule out degenerate cases, assume $3g+r>3$. Choose a base point $p$ on $\Sigma$ and let $S$ be a generating set of the fundamental group $\pi_1(\Sigma,p)$ whose elements we identify with (homotopy classes of) loops based at $p$. We say the generating set $S$ is a \emph{simple generating set} if it consists of simple loops that are pairwise non-homotopic and disjoint except at the base point. Note that the standard generating set for a surface group is simple.  Also, any triangulation of $\Sigma$ with a single vertex $p$ is a simple generating set for $\pi_1(\Sigma,p)$.  Suppose that $S$ is a simple generating set. We denote the word length of an element $w$ in $\pi_1(\Sigma,p)$ with respect to $S$ by $|w|_S$. 

By a \emph{curve} on $\Sigma$ we mean a properly immersed closed curve (i.e. the image in $\Sigma$ of a proper immersion of the unit circle), not freely homotopic to a point or boundary component. By an \emph{arc} we mean the image of the unit interval under a proper immersion such that the two endpoints are on the boundary of the surface. We say two curves are homotopic if they are freely homotopic, and two arcs are homotopic if they are homtopic relative to the boundary. 

We define the intersection number, $\iota(w,\alpha)$, of a loop $w$ based at $p$ with a curve or arc $\alpha$ on $\Sigma\setminus\{p\}$ to be the minimum number of intersections points between transverse representatives of $\alpha$ and $w'$ as $w'$ runs over all loops homotopic to $w$ relative to the base point $p$. Although closely related to it, note that this is not the standard geometric intersection number on $\Sigma$ or $\Sigma\setminus\{p\}$. If $\lambda=\sum t_i\alpha_i$ is a weighted system of curves and arcs, i.e. $t_i\in\mathbb{R}_{+}$ and each $\alpha_i$ a curve or arc, then $\iota(w,\lambda)=\sum t_i\cdot\iota(w,\alpha_i)$. Our first theorem states that, for a simple generating set $S$, we can compute the word length of $w$ by its intersection number with a well-chosen set of curves and arcs on $\Sigma\setminus\{p\}$. We stress that this is a precise equality. 

\begin{thm}\label{currentloop} 
Let $S$ be a simple generating set for $\pi_1(\Sigma,p)$. There exists a collection of weighted curves and arcs $\lambda_S$ on $\Sigma\setminus\{p\}$ such that $\iota(w,\lambda_S)=|w|_{S}$ for all $w\in\pi_1(\Sigma,p)$.
\end{thm}

The collection $\lambda_S$ of (weighted) curves and arcs will be constructed explicitly. Although this is not the main focus of the paper, Theorem \ref{currentloop} yields in particular a linear time algorithm to compute the word length in $\pi_1(\Sigma,p)$ with respect to a simple generating set: a word is of shortest length (i.e written with the minimal number of generators possible) if and only if it does not form any bigons with $\lambda_S$ (see Proposition \ref{bigonprop} for the precise statement).

Continuing with our results, note that instead of elements of the fundamental group, one can consider free homotopy classes of (closed, immersed) curves on $\Sigma$. To that end, recall that a free homotopy class $[\gamma]$ of a closed curve $\gamma$ can be identified with a conjugacy class, again denoted $[\gamma]$, of a certain element in $\pi_1(\Sigma,p)$. The word length of the curve $\gamma$, denoted $\ell_S(\gamma)$, with respect to the generating set $S$ is defined as 
$$\ell_S(\gamma)=\min\{\vert\delta\vert_S\,:\,\delta\in[\gamma]\}. $$ 
The intersection number of a (free homotopy class of a) curve $\gamma$ and an arc or curve $\alpha$ is defined as usual as the minimum number of intersections of transverse representatives of $\alpha$ and $\gamma'$ as $\gamma'$ runs over all curves freely homotopic to $\gamma$. We note that this is the usual geometric intersection number of curves on $\Sigma$. In this setting: 

\begin{thm}\label{currentfree}
Let $S$ be a simple generating set for $\pi_1(\Sigma,p)$ and $\lambda_S$ the collection of weighted curves and arcs provided by Theorem \ref{currentloop}. Then $\iota(\gamma,\lambda_S)=\ell_S(\gamma)$ for all curves $\gamma\subset\Sigma$.
\end{thm}

\noindent As we will see below (Corollary \ref{unique}), the collection of weighted curves and arcs $\lambda_S$ is unique with this property. 

Theorem \ref{currentfree} has an immediate consequence which again shows how special simple generating sets are: \emph{the word length, with respect to a simple generating set $S$, of the $k^{th}$-power of a curve is the $k^{th}$-multiple of the word length of the curve, that is $\ell_S(\gamma^k)=k\cdot\ell_S(\gamma)$}. It is easy to give examples to show this is not true for general generating sets. 

Another comment, more central to the results of this paper, is that Theorem \ref{currentloop} implies that the length function $\ell_S$ on the set of curves on $\Sigma$ with respect to a simple generating set $S$ extends naturally to the space of currents on $\Sigma$. To be precise, let $\mathcal{C}(\Sigma)$ denote the space of currents and recall that the set of all curves on $\Sigma$ is a subset of $\mathcal{C}(\Sigma)$ and the geometric intersection number of curves extends to an intersection form on the space of currents. In particular, considering the collection of curves and arcs $\lambda_S$ as a current, we get that the function $\mu\mapsto\iota(\lambda_S,\mu)$ is a homogenous, continuous extension of the word length to $\mathcal{C}(\Sigma)$. Here homogenous means that $\ell_S(t\mu)=t\cdot\ell_S(\mu)$ for all $t\in\mathbb{R}_{+}$. We have:

\begin{cor}
Let $S$ be a simple generating set for $\pi_1(\Sigma,p)$. The length function $\ell_S$ extends to a homogenous continuous function $\ell_S: \mathcal{C}(\Sigma)\to\mathbb{R}_{+}$. \qed
\end{cor} 

\noindent\textbf{Counting Curves.}
We now discuss an application of Theorem \ref{currentfree}. Being able to see the length function as given by intersection with a current allows us to count curves of bounded word length. Before making this precise, recall that combining results of \cite{MM2} and \cite{ES} one obtains that if $\rho$ is a metric of non-positive curvature on $\Sigma$, then for any immersed curve $\gamma_0$ 
\begin{equation}\label{metriclim}
\lim_{L\to\infty}\frac{\vert\{\gamma\in\MCG(\Sigma)\cdot\gamma_0\,:\,\ell_{\rho}(\gamma)\leq L\}\vert}{L^{6g-6+2r}}
\end{equation}
exists and is positive. 
Here $\MCG(\Sigma)\cdot\gamma_0$ denotes the mapping class group orbit of $\gamma_0$ and $\ell_{\rho}(\gamma)$ the length of a geodesic in the homotopy class of $\gamma$. As an application of Theorem \ref{currentfree} we will prove that limit \eqref{metriclim} also exists when replacing $\ell\rho$ with $\ell_S$ for a simple generating set $S$. 
This question was motivated by work of Chas, specifically by her empirical results concerning curves on the one-holed torus and their word lengths. It was she who first investigated the precise behavior of the growth of the number of curves of a given type with bounded word length (see, among others, \cite{Cha}). Chas conjectured that the limit \eqref{metriclim} exists also when considering word length and that it is closely related to the limit above (see Conjecture 1 of \cite{Cha}), which our theorem verifies: 

\begin{thm}\label{count}
Let $\Sigma$ be a surface of genus $g$ and with $r$ boundary components, where $3g+r>3$. Let $S$ be a simple generating set for $\pi_1(\Sigma,p)$ and $\gamma_0$ a closed immersed curve on $\Sigma$. Then 
$$\lim_{L\to\infty}\frac{\vert\{\gamma\in\MCG(\Sigma)\cdot\gamma_0\,:\,\ell_S(\gamma)\leq L\}\vert}{L^{6g+2r-6}}=C$$
exits and is positive. Moreover, $C=C_{\gamma_0}\cdot \mu_{Thu}(\{\nu\in\ML\,:\,\ell_S(\nu)\leq1\})$ where $C_{\gamma_0}>0$ is a constant depending only on $\gamma_0$ and $\mu_{Thu}$ is the Thurston measure on $\ML(\Sigma)$, the space of measured laminations of $\Sigma$. 
\end{thm}

\noindent The work of Chas suggests moreover that the ratios of the constants $C_{\gamma_0}$ are rational numbers (again see Conjecture 1 of \cite{Cha}). In the case of the one-holed torus and orbits of curves with small self-intersection number, she provides strong numerical evidence for this, and related, conjectures. The methods we use in this paper provide no explicit constants.      

It should perhaps be pointed out that there is, in general, no algorithm to determine whether or not two words in $S$ are in the same mapping class group orbit. Still, the construction of the current $\lambda_S$ lets us count the number of curves in each orbit as in the theorem above. 

It is an interesting question to determine if Theorem \ref{count} still holds for more general generating sets. While it is possible that some of the ideas of this paper still apply, the precise methods can not be used as is. For example, consider the one-holed torus and let $a,b$ be the standard generators of its fundamental group. Then the word length with respect to the generating set $\{a, b, a^{10}\}$ is not given by any current. 

\begin{remark}
While writing this paper, Parlier and Souto came up with an alternative argument to obtain Theorem \ref{count}. We later combined the insights of their work and this paper to prove that Theorem \ref{count} also holds for general generating sets (see \cite{EPS}). However, Theorems \ref{currentloop} and \ref{currentfree} cannot be generalized beyond simple generating sets.  
\end{remark}


The paper is organized as follows. In Section \ref{background} we will give the necessary background and definitions needed. In Section \ref{sectionproof} we construct the collection of curves and arcs $\lambda_S$ for a simple generating set $S$ and prove Theorems \ref{currentloop} and \ref{currentfree}. Finally, in Section \ref{counting} we give some background on currents and prove Theorem \ref{count}.\\

\noindent{\bf Acknowledgements.} I am very grateful to Moira Chas, Hugo Parlier, and Juan Souto for their help and encouragement throughout this project. I also thank Juan Souto for pointing out Corollary \ref{unique} to me. Also, I would like to thank the University of Fribourg for giving me the opportunity to spend a semester there, and to acknowledge support from Swiss National Science Foundation grant number PP00P2\_128557. Lastly, I thank the referees for their useful comments. 


\section{Background }\label{background}


Throughout the paper $\Sigma$ will be a connected, orientable, compact surface of genus $g$ with $r$ boundary components such that $3g+r>3$. Fixing once and for all a base point $p$, we will always consider the fundamental group $\pi_1(\Sigma, p)$ based at $p$. If $w$ and $w'$ are loops based at $p$ and homotopic relative to $p$ we write $w\sim_p w'$. We identify the elements of $\pi_1(\Sigma, p)$ with homotopy classes of loops based at $p$. 

Recall that an \emph{arc} on a surface $\Sigma$ is an immersed arc whose endpoints are on its boundary $\partial\Sigma$ and whose interior is disjoint from  $\partial\Sigma$. That is, an arc is the image of a proper immersion of the unit interval to $\Sigma$, $i:I\to\Sigma$, such that $i^{-1}(\partial\Sigma)=\{0,1\}$. We further assume arcs to be non-trivial, i.e. not homotopic (relative boundary) to a point. Also, a \emph{curve} on $\Sigma$ is a closed immersed curve, i.e. the image of a proper immersion $S^1\to\Sigma$,  that is essential and non-peripheral (i.e. not freely homotopic to a point or boundary component).

A subset $S=\{g_1,\ldots,g_n\}\subset\pi_1(\Sigma,p)$ is a generating set for $\pi_1(\Sigma,p)$ if any element $w\in\pi_1(\Sigma,p)$ can be written as $w=g_{n_1}^{\alpha_1} g_{n_2}^{\alpha_2}\cdots g_{n_l}^{\alpha_l}$ for some $g_{n_i}\in S$ and $\alpha_i\in\mathbb{Z}$. We say that $S$ is a \emph{simple generating set} if all elements in $S$ are simple loops that are pairwise disjoint (except for at the base point). Note that, in particular, the standard minimal generating set for a genus $g$ surface, $S=\{a_1, b_1, a_2, b_2\ldots, a_g, b_g\}$ where the product of the commutators satisfy $[a_1,b_1]\cdots[a_{g},b_g]=1$, is a simple generating set. Also, as mentioned in the introduction, any one vertex triangulation on $\Sigma$ gives a simple generating set (this time non-minimal). In particular, the number of simple generating sets, up to automorphisms of $\pi_1(\Sigma,p)$, is super exponential in the complexity of the surface (see \cite{Pen} or \cite{DP} for more details).  

We define the \emph{word length} of $w\in\pi_1(\Sigma,p)$ with respect to a generating set $S$, denoted by $\vert w\vert_S$, to be the minimum number of generators needed to write the word $w$. That is, 
$$\vert w\vert_S=\min\{\vert\alpha_1|+\cdots+|\alpha_l\vert\, : \, w'= g_{n_1}^{\alpha_1} g_{n_2}^{\alpha_2}\cdots g_{n_l}^{\alpha_l} \mbox{ and } w'\sim_p w\}.$$

Let $\lambda$ be a properly immersed curve or arc (or union of such) on $\Sigma\setminus\{p\}$ and let $|w\cap\lambda|$ denote the number of intersection points of transverse representatives of $w$ and $\lambda$. We define the (geometric) intersection number between $w$ and $\lambda$, $\iota(w,\lambda)$, to be the minimum number of intersection points between transverse representatives of $\lambda$ and $w'$ as $w'$ runs over all loops homotopic to $w$ relative $p$. That is:
$$\iota(w,\lambda)=\min\{\vert w'\cap\lambda\vert\,:\,w'\sim_p w\mbox{ and in general position}\}.$$
Note that both $\vert w\vert_S$ and $\iota(w,\lambda)$ only depend on the homotopy class of $w$.

Given two curves $\gamma, \gamma'$ in $\Sigma$ we write $\gamma\sim\gamma'$ if $\gamma$ and $\gamma'$ are (freely) homotopic, and denote the (free) homotopy class of $\gamma$ by $[\gamma]$. Each curve $\gamma$ can be identified with a (non-unique) element in $\pi_1(\Sigma,p)$ by identifying it with a curve through $p$ that is homotopic to $\gamma$ and viewing it as a loop based at $p$. Hence we can write $\gamma$ as a word in the generators in $S$. Suppose $w$ is a word representing $\gamma$. Then the curve given by conjugating $w$ by any other word $u$ gives a curve that is freely homotopic to $\gamma$ as well. In fact, conjugacy classes of elements in $\pi_1(\Sigma,p)$ correspond to homotopy classes of curves on $\Sigma$. The word length, $\ell_S(\gamma)$, of (the homotopy class of) $\gamma$ with respect to a generating set $S$ is defined as the minimum word length among all elements in the conjugacy class corresponding to $\gamma$. That is, 
$$\ell_S(\gamma)=\min\{\vert \delta\vert_S\,: \, \delta\in[\gamma]\}.$$

We define the (geometric) intersection number between $\gamma$ and another curve or arc (or union of such) $\lambda$ to be the minimum number of intersections between transverse representatives of $\lambda$ and $\gamma'$ as $\gamma'$ runs over all curves in the free homotopy class of $\gamma$, that is:  
$$\iota(\gamma,\lambda)=\min\{\vert \gamma'\cap\lambda\vert\,:\,\gamma'\sim \gamma\mbox{ and in general position}\}.$$
As above, $\ell_S(\gamma)$ and $\iota(\gamma,\lambda)$ only depend on the free homotopy class $[\gamma]$. 


\section{Proof of Theorems \ref{currentloop} and \ref{currentfree}}\label{sectionproof}

\subsection{Constructing the collection of curves and loops $\lambda_S$}\label{current}

Consider the surface $\Sigma\setminus S$ obtained by cutting $\Sigma$ along the loops corresponding to the generators. Since $S$ generates $\pi_1(\Sigma,p)$ we obtain a disjoint union of simply connected or possibly once-holed polygons. Note that all vertices are identified to a single point, namely the base point $p$, and the edges correspond to the generators.  We start by constructing arcs in each polygon, which then will be connected in such a way to create a union of arcs and curves on $\Sigma$: 

\begin{enumerate}[(i)]
\item Mark two points, away from the vertices, on each edge of each polygon\\
\item For each once-holed polygon $P$, connect each marked point with a simple arc in the interior of $P$ to the hole, such that all arcs are pairwise disjoint (see Figure \ref{puncture})\\
\item For each simply connected polygon $P$ with an even number of sides, and for each pair of opposite edges, connect the two points on one edge with two simple disjoint arcs in the interior of $P$ to the two points on the opposite edge (see Figure \ref{even})\\
\item For each simply connected polygon $P$ with an odd number of sides, connect the two points on each side with two simple arcs in the interior of $P$, disjoint from each other, to a point on each of the two edges incident to the vertex opposite this side (see Figure \ref{odd})\\
\item Glue the polygons back together according to the side pairing given by the cutting $\Sigma\setminus S$. When gluing two polygons together along a side $e$ we have to decide how to connect the two endpoints on $e$ of the arcs in one polygon to the two endpoints of the arcs in the other: 
 \begin{itemize}
 \item Case 1: If at least one of the polygons has a hole or is even sided, connect the arcs pairwise without creating a new intersection.
 \item Case 2: If both polygons are simply connected and odd sided, connected them by creating a crossing.
 \end{itemize}
See Figures \ref{cross} and \ref{glue}. 
\end{enumerate}

\begin{subfigures}
\begin{figure}[h]
\includegraphics[width=8cm, height=11.2cm]{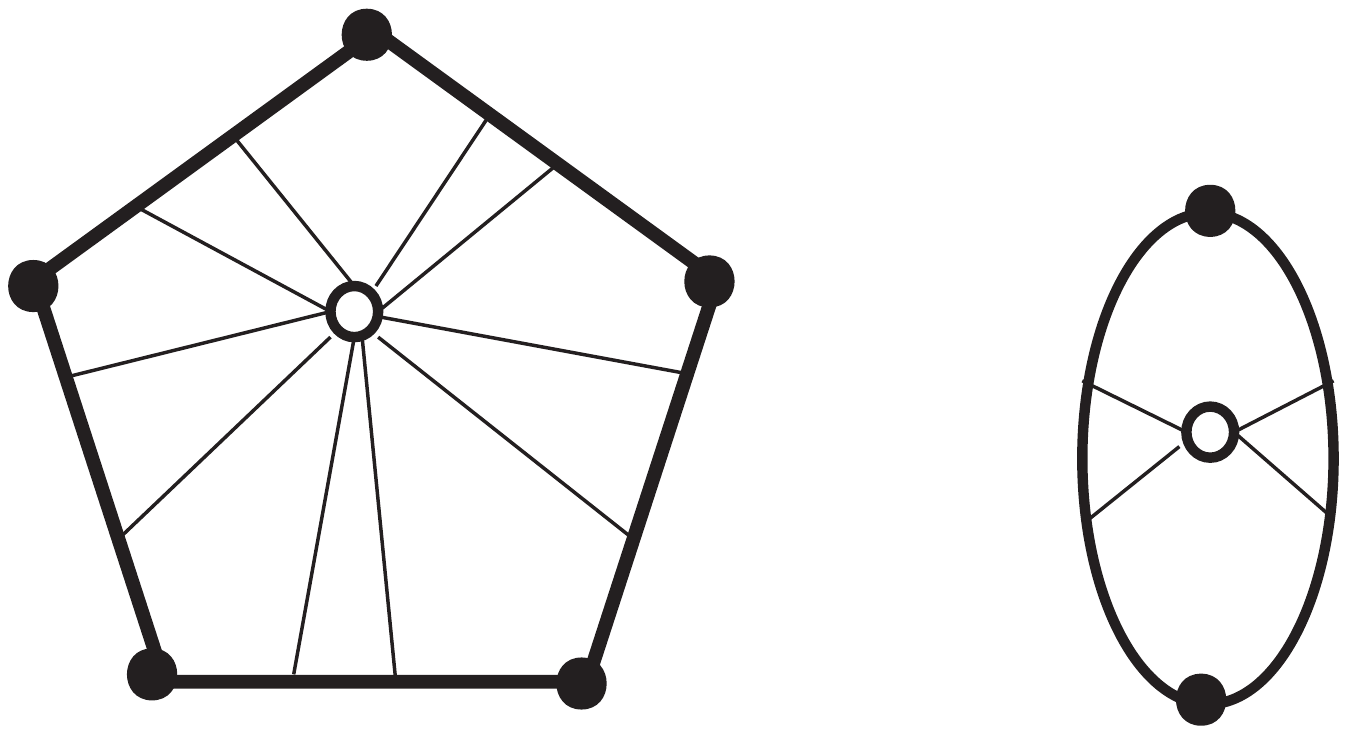}
\vspace{-7.5cm}
\caption{Arcs in one-holed polygons}
\label{puncture}
\end{figure}

\begin{figure}[h]
\hspace{2cm}
\includegraphics[width=8cm, height=11cm]{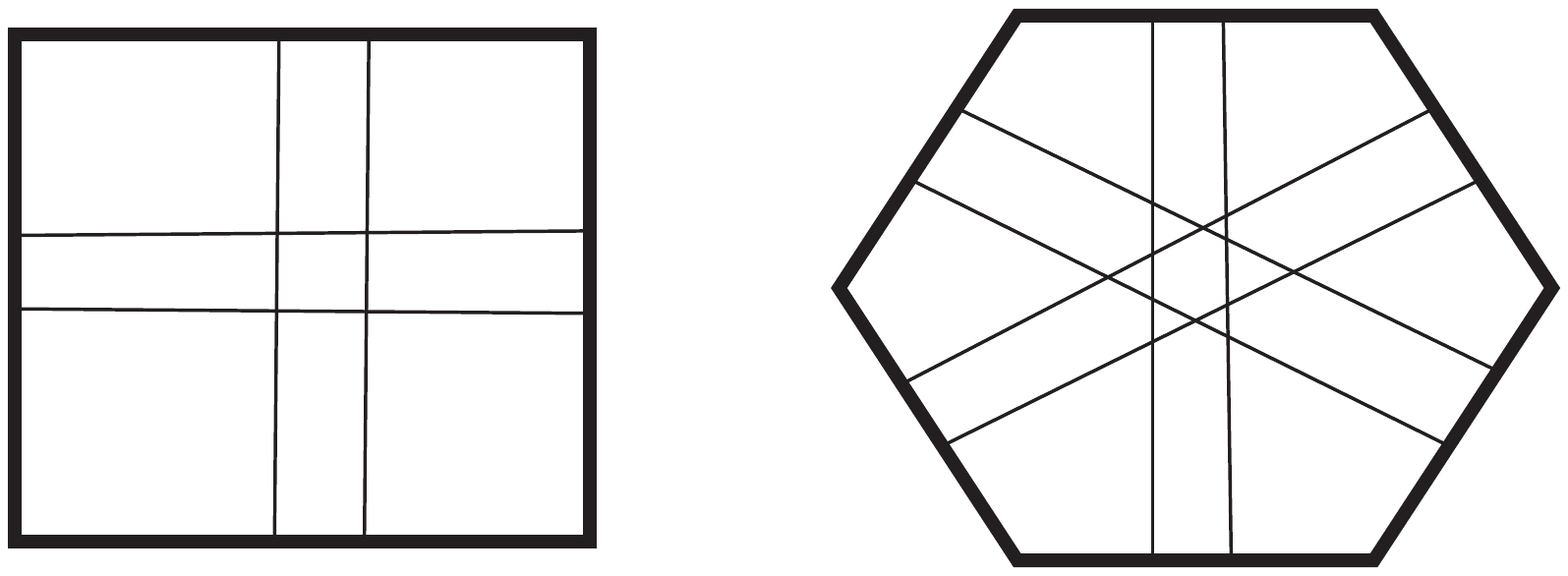}
\vspace{-7.8cm}
\caption{Arcs in simply connected even sided polygons}
\label{even}
\end{figure}

\begin{figure}[h]
\hspace{0.5cm}
\includegraphics[width=8cm, height=11cm]{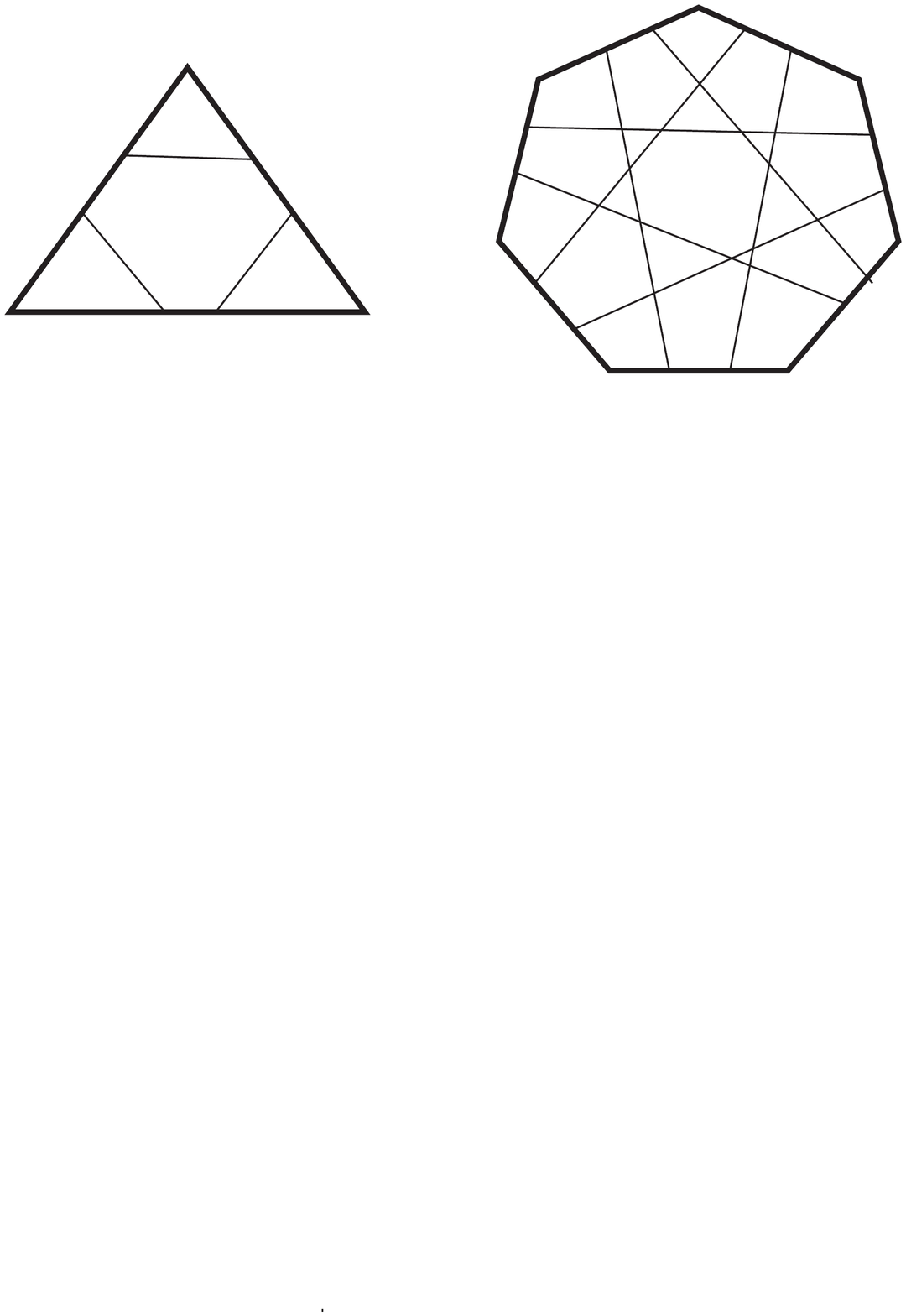}
\vspace{-7.5cm}
\caption{Arcs in simply connected odd sided polygons}
\label{odd}
\end{figure}

\begin{figure}[h]
\hspace{0.5cm}
\includegraphics[width=10cm, height=9cm]{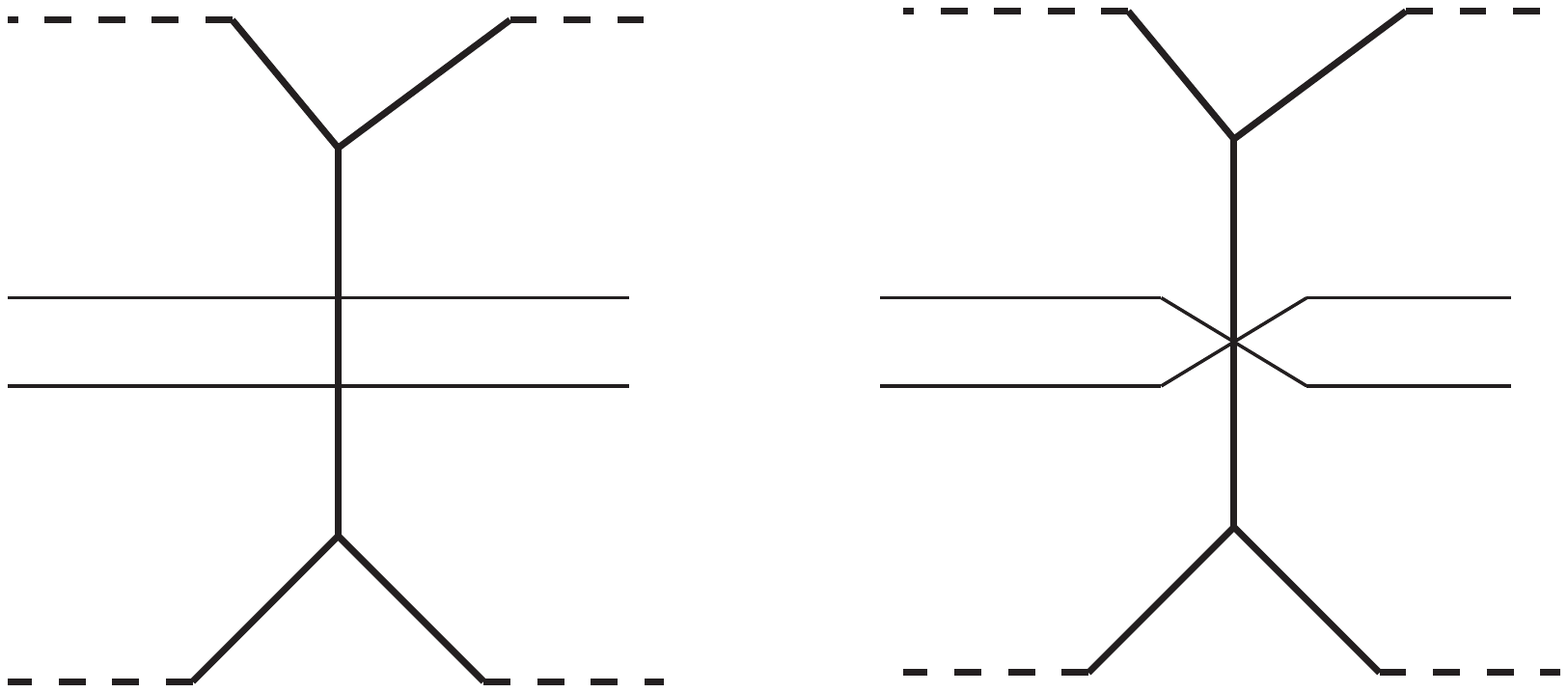}
\vspace{-5.3cm}
\caption{The two cases when gluing two polygons along a side: Case 1 on the left,  Case 2 on the right.}
\label{cross}
\end{figure}

\begin{figure}[h]
\hspace{10cm}
\includegraphics[width=12.5cm, height=8.2cm]{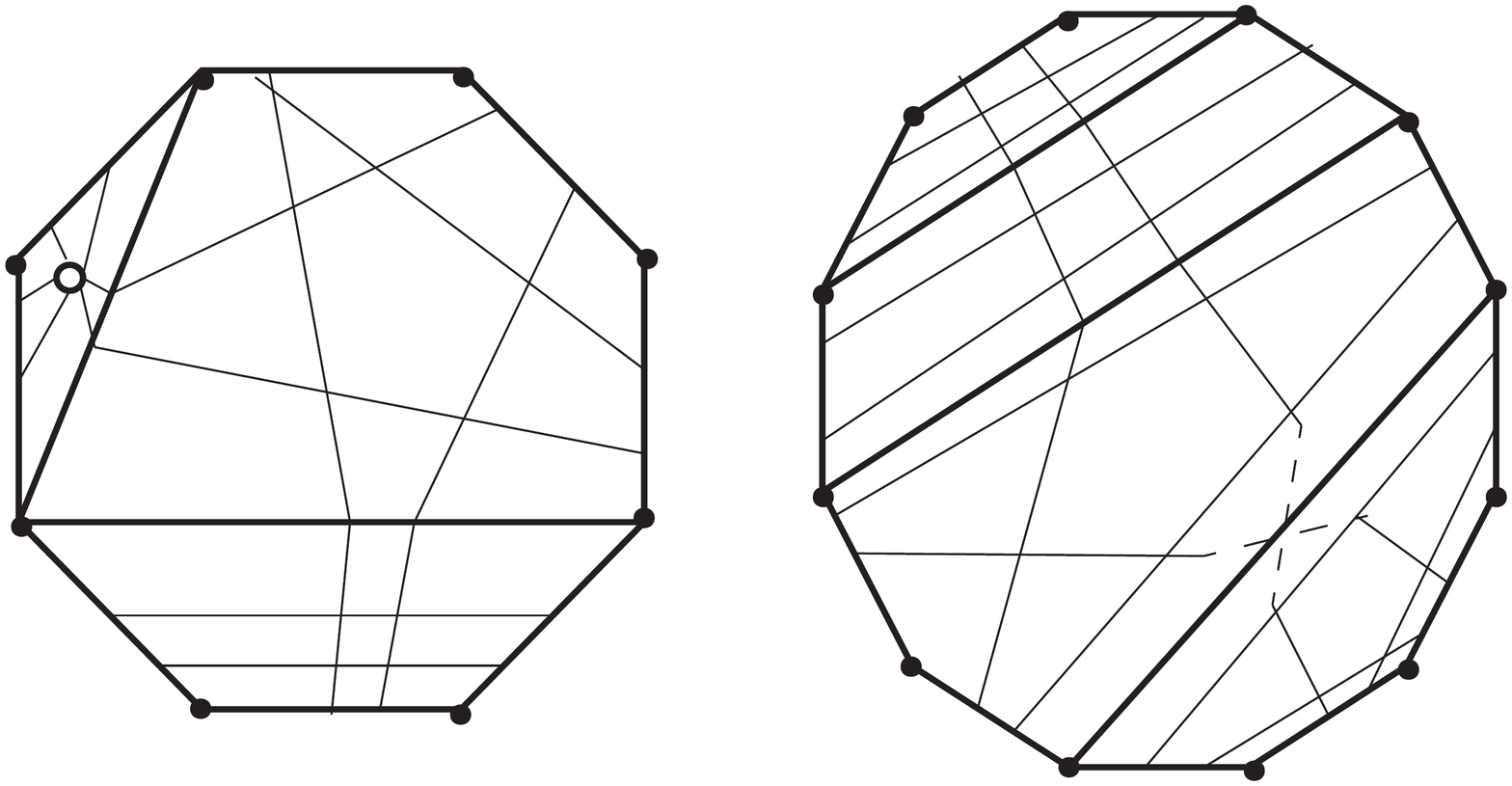}
\vspace{-1.5cm}
\caption{Connecting arcs when gluing polygons. Generators in thick lines and arcs in thin lines. Note the extra crossing in dashed lines.}
\label{glue}
\end{figure}

\end{subfigures}

\noindent Denote the resulting union of arcs and curves by $\lambda'_S$. Note that, by construction, $\lambda'_S$ intersects each generator exactly twice. In order to get a nicer representation, we multiply each curve and arc in $\lambda'_S$ by the coefficient 1/2 to obtain $\lambda_S$. We refer to $\lambda_S$ as the collection of (weighted) curves and arcs \emph{associated to} the simple generating set $S$. By the construction of $\lambda_S$ the following lemma is immediate: 

\begin{lemma}\label{atleast}
Suppose $S$ is a simple generating set for $\pi_1(\Sigma,p)$ and $\lambda_S$ the associated collection of curves and arcs. If $w=g_{n_1}^{\alpha_1} g_{n_2}^{\alpha_2}\cdots g_{n_l}^{\alpha_l}\in\pi_1(\Sigma,p)$ where $g_{n_i}\in S$ and $\alpha_i\in\mathbb{Z}$, then 
$$\vert\lambda_S\cap w\vert=\vert\alpha_1|+\cdots+\vert\alpha_n\vert.$$ 
In particular, $\iota(w,\lambda_S)\leq \vert w\vert_S$ for all $w\in\pi_1(\Sigma,p)$.\qed
\end{lemma}

\subsection{The proof}\label{proofs} 

To prove Theorems \ref{currentloop} and \ref{currentfree} we need to show that the word length of (a homotopy class of) a loop or curve with respect to a simple generating set $S$ is given by its intersection number with $\lambda_S$. The main technical step is showing that, when considering a word either as a loop or a curve in $\Sigma$, the presence of bigons with $\lambda_S$ is equivalent to being able to reduce the number of generators in the word. To make this precise, let us say that a word $w$ in $S$ is \emph{shortest} if its length (i.e. the number of generators appearing in $w$) is not longer than the length of any other word representing the same element as $w$. Equivalently, a word is shortest if its length realizes the word length of $w$. Accordingly, we say a word is \emph{cyclically shortest} if it is not longer than any other word representing the same conjugacy class. Finally, we need to define a bigon: First consider elements $w\in\pi_1(\Sigma,p)$ as loops. That is, a path $w: [0,1]\to\Sigma$ such that $w(0)=w(1)=p$. We say $w$ and an arc or curve $\alpha\subset\Sigma\setminus{p}$ form a bigon if there exist an interval $I\subset(0,1)$ and an arc $a\subset\alpha$ such that $w(I)$ and $a$ have the same endpoints and when concatenated they form a homotopically trivial curve. When we consider $w$ as (a free homotopy class of) a closed curve $w: \mathbb{S}^1\to\Sigma$ we modify the definition by allowing $I\subset\mathbb{S}^1$. Bigons will be important to us because if we have a curve (or arc) $\alpha$ and a collection of curves and arcs $\lambda$, then $\vert\alpha\cap\lambda\vert=\iota(\alpha,\lambda)$ if and only if $\alpha$ and $\lambda$ form no bigons.  

\begin{prop}\label{bigonprop}
Let $S$ be a simple generating set for $\pi_1(\Sigma,p)$ and $\lambda_S$ the associated collection of curves and arcs. Then:
\begin{enumerate}[(i)]
\item $w\in\pi_1(\Sigma,p)$ is shortest if and only if $w$, when considered as a loop, does not form a bigon with $\lambda_S$. \\
\item $w\in\pi_1(\Sigma,p)$ is cyclically shortest if and only if $w$, when considered as a curve, does not form a bigon with $\lambda_S$.
\end{enumerate}
\end{prop}

Assuming Proposition \ref{bigonprop} for now, we will prove Theorems \ref{currentloop} and \ref{currentfree}. 

\begin{proof}[Proof of Theorems \ref{currentloop} and \ref{currentfree}]
Let $S$ be a simple generating set for $\pi_1(\Sigma,p)$ and $\lambda_S$ the associated collection of curves and arcs. We start with considering Theorem \ref{currentloop}. Let $w\in\pi_1(\Sigma,p)$ and $w'$ a shortest word with respect to $S$ which defines the same element as $w$, i.e. $w'\sim_p w$. Then, by definition, $\vert w\vert_S=\vert w'\vert_S$ and $\iota(\lambda_S,w)=\iota(\lambda_S,w')$. By Proposition \ref{bigonprop} $w'$ forms no bigons with $\lambda_S$. Hence $\vert\lambda_S\cap w'\vert=\iota(\lambda_S,w')$. But by Lemma \ref{atleast}, since $w'$ is shortest,  $\vert\lambda_S\cap w'\vert=\vert w'\vert_S$. Therefore we have $\iota(\lambda_S,w)=\iota(\lambda_S,w')=\vert\lambda_S\cap w'\vert=\vert w'\vert_S=\vert w\vert_S$, as desired. 

Theorem \ref{currentfree} is proved similarly.
\end{proof}

We point out the following fact that also results from Proposition \ref{bigonprop}:

\begin{cor}
Let $S$ be a simple generating set for $\pi_1(\Sigma,p)$ and $\lambda_S$ the associated collection of curves and arcs. Then each curve and arc in $\lambda_S$ is essential. Moreover, $\iota(g,\lambda_S)=1$ for all $g\in S$. \qed
\end{cor}

It remains to prove Proposition \ref{bigonprop} which we will do below. 

\subsection{Looking for bigons}\label{bigon} 

The results of this section hold for any simple generating set for any surface $\Sigma$ satisfying $3g+r>3$. However, first we will assume that $\Sigma$ is closed and return to the general case at the end of the section. 

Consider tiling the universal cover $\tilde{\Sigma}$ of $\Sigma$ by the lifts of the polygons in $\Sigma\setminus S$. Call this tiling $T$. Note that the 1-skeleton of $T$ corresponds to the lifts of $S$, which we denote $\tilde S$, and every vertex of $T$ is a lift of the base point $p$. Let $\tilde\lambda_S$ be the lift of $\lambda_S$ to $\tilde{\Sigma}$. Note that if $\mathcal{P}$ is a union of polygons in $T$, every edge in $P$ is adjacent to at most two polygons. We call the edges adjacent to only one polygon the \emph{exterior edges} of $\mathcal{P}$ and the rest the \emph{interior edges}. By a \emph{path $\tilde p$ in $\tilde S$} we mean a path that traverses edges of the graph $\tilde S$ that is either infinite or starts and ends at vertices, and the length of such a path is defined as the number of edges it traverses. Note that each path corresponds to a word in $S$ and the length of the path corresponds to the length of the word. The purpose of this section is to show that if a path $\tilde p$ in $\tilde S$ forms a bigon with $\tilde\lambda_S$, then there is a subpath of $\tilde p$ that can be replaced by a shorter path, and hence the corresponding word is not shortest. 

We will prove the above fact using a variation of Dehn's algorithm (see \cite{Deh} or \cite{Sti}), used to solve the word problem. Recall the proof of Dehn's algorithm in, for example, \cite[Chapter 6]{Sti}: We tile $\tilde\Sigma$ by the fundamental domain of $\Sigma$ and consider a path $\tilde p$ in the 1-skeleton of $T$, corresponding to a word $w$. If $\tilde p$ is a closed loop, then there must be a tile $D$ in $T$ of which $\tilde p$ traverses more than half its edges, in succession. Call this subpath $\tilde p'$. Then $\tilde p$ can be replaced with a shorter loop defining a word homotopic to $w$ by replacing the subpath $\tilde p'$ with a path traversing the complementary edges in $D$. This way, whenever we have a loop (i.e. a word representing the trivial word), we can chip away the tiles in $T$ until we are left with a trivial loop.  

Returning now to our setting with a tiling $T$ of $\tilde\Sigma$ given by $\Sigma\setminus S$, suppose there is a path in $\tilde S$ that forms a bigon with $\tilde\lambda_S$. Say the two sides of the bigon are $\tilde p$ and $\tilde a$ for some path $\tilde p$ in $\tilde S$ and some arc $\tilde a$ of a leaf $\tilde\alpha\subset\tilde\lambda_S$. We define a sequence of nested regions $\mathcal{P}_0\subset \mathcal{P}_1\subset \mathcal{P}_2\subset\cdots$ inductively. Let $\mathcal{P}_0$ be the union of the polygons in $T$ that $\tilde a$ intersects, and let $\mathcal{P}_{i}$ be the union of all polygons that have at least one vertex in $\mathcal{P}_{i-1}$. Moreover, let $\mathcal{B}_i$ denote the boundary of $\mathcal{P}_i$ for each $i\geq0$, set $\mathcal{A}_0=\mathcal{P}_0$, and define $\mathcal{A}_i=\mathcal{P}_i\setminus\mathcal{P}_{i-1}$ for each $i\geq1$. 

Suppose the path $\tilde p$ intersects $\mathcal{B}_k$ but does not intersect $\mathcal{B}_l$ for any $l>k$. We will show that if $k\neq 0$, then $\tilde p$ must have a subpath $\tilde p'$ that traverses at least half the edges, in succession, of a polygon (or possibly the union of two adjacent polygons) $P$ in $\mathcal{A}_k$ (see Lemma \ref{Dehn} below). Replacing the subpath $\tilde p'$ with the path $\tilde p''$ that traverses the complementary edges in $P$ gives a path of at most the same length. Note that the bigon the new path bounds with $\tilde a$ no longer contains $P$ in its interior. Repeating this argument one can assume that a path $\tilde p$ corresponding to a shortest word lies in $\mathcal{P}_0$.  Finally, we will show that (Lemma \ref{path}) if a path $\tilde p$ in $\mathcal{P}_0$ forms a bigon with $\tilde a$, then $\tilde p$ traverses more than half of the exterior edges, in succession, of $\mathcal{P}_0$. Again, this implies that $\tilde p$ could not have corresponded to a shortest word. 

Before proving the mentioned lemmas, we need the following: 

\begin{lemma}\label{base}
Every vertex in $\mathcal{P}_0$ has valence at most 4 in $\mathcal{P}_0$ (i.e. each vertex has at most 4 incident edges that lie in $\mathcal{P}_0$). 
\end{lemma}

\begin{proof} 
Let $v$ be a vertex of  $\mathcal{P}$ that is $k$-valent in  $\mathcal{P}$, i.e. it has $k$ edges incident to it that lie in  $\mathcal{P}$. Note that $\tilde a$ intersects each interior edge of $\mathcal{P}_0$ and $v$ has at most two exterior edges incident to it. Suppose $k>3$. Then $v$ has at least two incident interior edges. Since $\tilde a$ intersects every interior edge, any two such adjacent edges (i.e. edges that bound the same polygon) must be edges of a triangle, due to the construction of $\lambda_S$. Hence, if $k=3+l$ there are $l$ triangles incident to $v$. Suppose $l>2$. Then there are two triangles $\Delta_1$ and $\Delta_2$ that share an edge $e$ and $e$ is adjacent to $v$. For $i=1,2$, let $c_i$ and $d_i$ be the two other edges of $\Delta_i$, where $c_i$ is also an interior edge in $\mathcal{P}$ adjacent to $v$. Hence $\tilde a$ is built from an arc in $\tilde\lambda_S$ with one endpoint on $c_1$ and one on $e$ and another arc with one endpoint on $c_2$ and one on $e$. But when the two triangles are glued together along side $e$ these two arcs do not get glued together due to the rules for constructing $\lambda_S$ (see Figure \ref{triangle}), a contradiction. Hence $l\leq 1$ and $k\leq 4$. 
\end{proof}
\vspace{-1cm}
\begin{figure}[h]
\hspace{2cm}
\includegraphics[width=10cm, height=15cm]{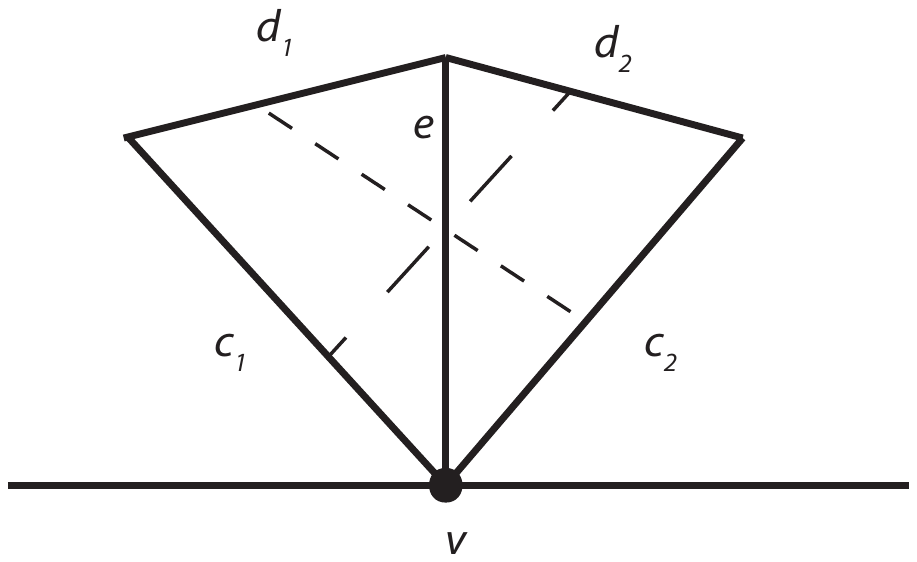}
\vspace{-11cm}
\caption{A vertex of valence 5. The arcs of $\tilde\lambda_S$ that cross the common side of the two triangles drawn in dashed lines.}
\label{triangle}
\end{figure}


\begin{lemma}\label{Dehn}
Suppose $\tilde p$ and $\tilde a$ form a bigon in $\tilde\Sigma$ where $\tilde p$ is a path in $\tilde S$ and $\tilde a$ is an arc of a leaf $\tilde\alpha\subset\lambda_S$. Suppose there exists $k>0$ such that $\tilde p$ intersects $\mathcal{B}_k$ but not $\mathcal{B}_l$ for any $l>k$. Then $\tilde p$ traverses at least half the (exterior) edges, in succession, of a polygon or a union of two adjacent polygons $P$ in $\mathcal{A}_k$. Moreover, none of the complementary edges of $P$ lie on $B_k$. 
\end{lemma}

\begin{proof}
Suppose $S$ is a simple generating set with $\vert S\vert=n\geq3$ (note that this is a trivial assumption when $\Sigma$ is closed, but will play a role in the general case). Let $T$ be the tiling of $\tilde\Sigma$ corresponding to the lifts of $\Sigma\setminus S$ and note that every vertex of $T$ is of valence at least 6. 

We claim that every vertex on $\mathcal{B}_i$ has valence at most 4 in $\mathcal{P}_{i}$. This is true for $i=0$ by Lemma \ref{base}. Suppose the claim holds for $\mathcal{B}_{i-1}$. Then, since every vertex in $T$ has valence $2n$ there are at least $2n-4$ polygons in $\mathcal{A}_i$ incident to each vertex on $\mathcal{B}_{i-1}$. It follows (since $n\geq 3$) that every polygon in $\mathcal{A}_i$ has at most one edge on $\mathcal{B}_{i-1}$. Therefore every polygon in $\mathcal{A}_i$ share an edge with at most 3 other polygons in $\mathcal{P}_i$, and these edges are consecutive. Hence, if $P_m$ is an $m$-gon in $\mathcal{A}_i$ it has at least $m-3$ consecutive edges on $\mathcal{B}_i$, and $m-2$ such edges if it has no edge on $\mathcal{B}_{i-1}$. If $v$ is a vertex on $\mathcal{B}_i$ incident to a triangle in $\mathcal{A}_i$ with an edge on $\mathcal{B}_{i-1}$ it has valence 4 in $\mathcal{P}_{i}$, and in all other cases it has valence 3 in $\mathcal{P}_{i}$. By induction, this proofs the claim.

Note that since each vertex $v$ on $\mathcal{B}_{k-1}$ has valence at most 4 in $\mathcal{A}_{k-1}$, $v$ has at least $2n-4\geq2$ incident edges  that are in $\mathcal{A}_k$ and not on $\mathcal{B}_{k-1}$. Hence there cannot be two adjacent polygons in $\mathcal{A}_k$ (i.e. polygons that share an edge) that both have an edge on $\mathcal{B}_{k-1}$. In particular, if an $m$-gon in $\mathcal{A}_{k}$ has only $m-3$ of its edges on $\mathcal{B}_k$ then the adjacent polygons in $\mathcal{A}_{k}$ have all but two of it's edges on $\mathcal{B}_k$.  

Now suppose $k>0$ is the largest integer such that $\tilde p$ intersects $\mathcal{B}_k$.  Then $\tilde p$ traverses an edge $e_1$ from $\mathcal{B}_{k-1}$ to $\mathcal{B}_{k}$ and a path (possibly just a vertex) of length $l\geq0$ on $\mathcal{B}_k$ and then an edge $e_2$ from $\mathcal{B}_k$ to $\mathcal{B}_{k-1}$. 

If $l=0$, i.e. the path is just a vertex on $\mathcal{B}_{k}$, then $e_1$ and $e_2$ are edges of the same polygon which must be a triangle. Hence $\tilde p$ traverses 2 edges of this triangle and we are done. 

So suppose $l\geq1$. Then $\tilde p$ traverses an edge on $\mathcal{B}_{k}$ of at least one polygon in $\mathcal{A}_{k}$. Let $P_m$ be the first such $m$-gon, for some $m\geq3$. The path traverses at least $m-3$ edges of $P_m$ lying on $\mathcal{B}_{k}$, and $m-2$ if it has no edge on $\mathcal{B}_k$ (which in particular has to be the case if $m=3$). If $e_1$ is also an edge of $P_m$, then $\tilde p$ traverses at least $m-2$ (and 2 if $m=3$) consecutive edges of $P_m$ and we have proved the proposition. Now suppose $e_1$ is not an edge of $P_m$. Then it must be an edge of a triangle $\Delta$ adjacent to $P_m$ which has an edge on $\mathcal{B}_{k-1}$. Then $P_m$ has no edge on $\mathcal{B}_{k-1}$, and hence $\tilde p$ traverses $m-2$ consecutive edges of $P_m$ (lying on $\mathcal{B}_k$). Let $P$ be the union of $\Delta$ and $P_m$ which has $m+1$ exterior edges. The path $\tilde p$ traverses $m-1$ successive exterior edges of $P$, and none of the complementary edges lie on $\mathcal{B}_k$.  
 \end{proof}

Lemma \ref{Dehn} implies that if $\tilde p$ represents a shortest word, then we can assume that $\tilde p$ is a path contained in $\mathcal{P}_0$. Next we show that even this path can be replaced with a shorter path defining the same word, and hence the path did not represent a shortest word. 

\begin{lemma}\label{path}
Let $\tilde p$ be a path contained in $\mathcal{P}_0$ that forms a bigon with $\tilde a$. Then $\tilde p$ traverses more than half the exterior edges of $\mathcal{P}_0$, in succession. 
\end{lemma}

\begin{proof}
We first consider the case when $\mathcal{P}_0$ is a single polygon, say an $m$-gon $P_m$ of $T$. If $m$ is even, $\tilde a$ is an arc between two opposite sides of the $m$-gon and hence $\tilde p$ must traverse at least $(m/2)+1$ edges of $P_m$. If $m$ is odd, then $\tilde a$ is an arc between an edge and one of the edges incident to the vertex opposite the first edge. Hence $\tilde p$, in order to form a bigon with $\tilde a$, must traverse either $(m+1)/2$ or $(m+3)/2$ edges of $P_m$, and these edges are successive. See Figure \ref{bigonsmall}.

\begin{figure}[h]
\hspace{0cm}
\includegraphics[width=8cm, height=12cm]{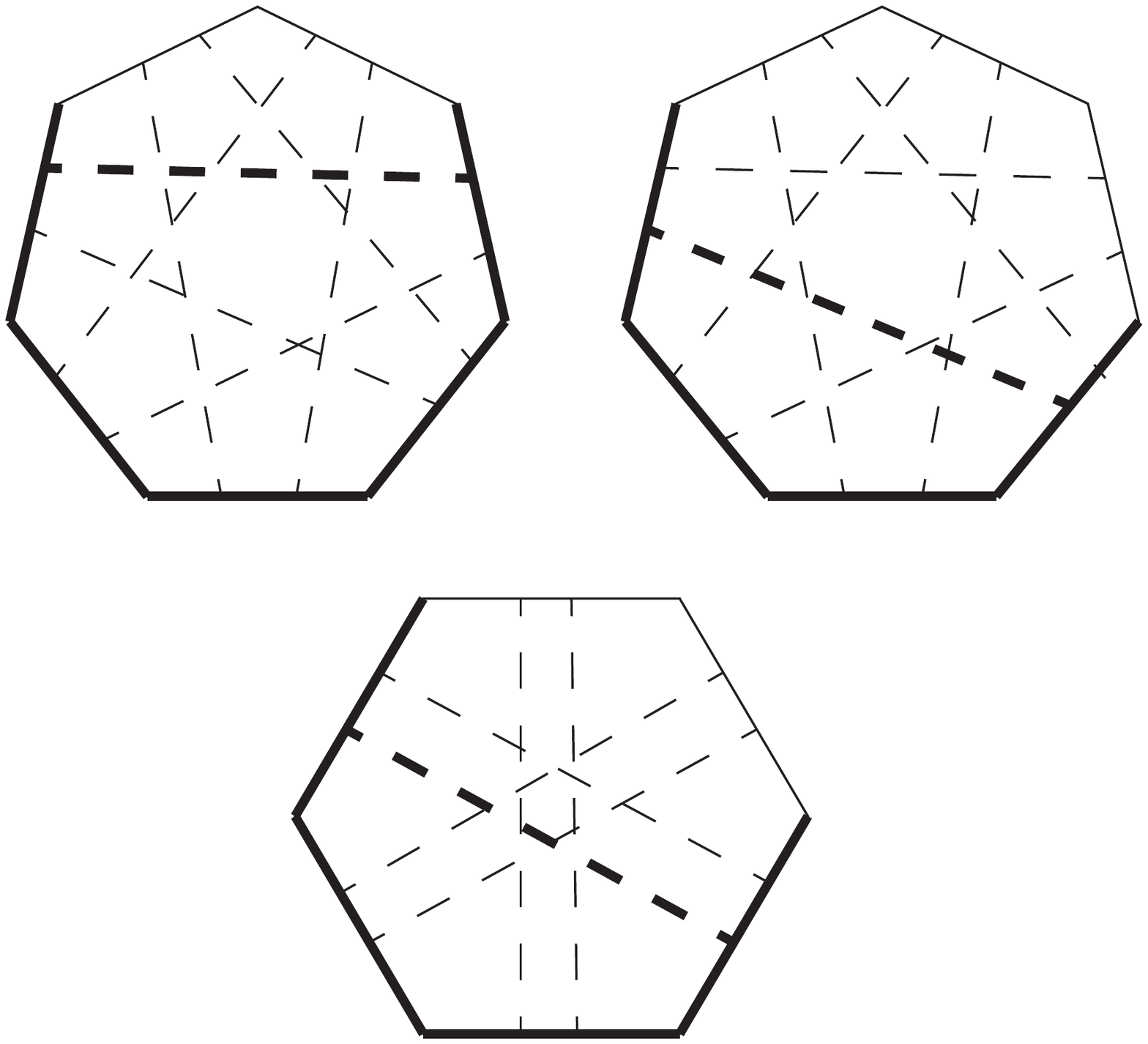}
\vspace{-3.5cm}
\caption{Example of bigons, drawn in bold. The solid lines represent generators (edges) and the dashed lines represent arcs of $\lambda_S$.}
\label{bigonsmall}
\end{figure}

Now we consider the case when $\mathcal{P}_0$ is the (connected) union of $k$ polygons. We can view $\mathcal{P}_0$ as a single polygon by ignoring the interior edges. Suppose $\tilde a$ has one of its endpoints on exterior edge $e_1$ and the other on exterior edge $e_2$. We claim that if $\mathcal{P}_0$ has an even number of exterior edges, then $e_1$ and $e_2$ are opposite each other in $\mathcal{P}_0$, when viewed as a single polygon; if it has an odd number of exterior edges, then $e_2$ is an edge incident to the vertex opposite $e_1$. Hence, applying the argument above, the path $\tilde p$ must traverse at least half the exterior edges of $\mathcal{P}_0$ in order to form a bigon with $\tilde a$. 

We prove the claim by induction on $k$. The base case, $k=1$, follows directly by the definition of $\lambda_S$. Suppose the claim holds for $k-1$, i.e. $\tilde a$ restricted to any connected union $\mathcal{P}$ of $k-1$ polygons is an arc with endpoints on exterior sides opposite each other if it is an even sided polygon, and on an exterior side and the exterior side incident to vertex opposite the first if it is odd sided. Consider $\mathcal{P}_0$, a union of $k$ polygons, and $\tilde a$ with endpoints on (exterior) edges $e_1$ and $e_2$ of $\mathcal{P}_0$. Let $P_m$ be the $m$-sided polygon in $\mathcal{P}_0$ that has $e_2$ as one of its edges. Let $\mathcal{P}=\mathcal{P}_0\setminus P_m$, i.e. the union of the other $k-1$ polygons. Note that $\mathcal{P}$ is connected. Suppose it has $n$ exterior edges. Then $\mathcal{P}_0$ has $m+n-2$ exterior edges.  Consider $\tilde a$ restricted to $\mathcal{P}$. It has its endpoints on edge $e_1$ and some exterior edge $e$ of $\mathcal{P}$. Note that the edge $e$ is also an edge of $P_m$ (and an interior edge of $\mathcal{P}_0)$. 

Now, suppose $n$ is even. By the induction hypothesis, $e$ is opposite $e_1$ in $\mathcal{P}$. 

Assume $m$ is also even. Then, since $\tilde a$ is an arc of a leaf in $\tilde\lambda_S$, $e_2$ is opposite $e$ in $P_m$. We have $m+n-2$ even, and $e_1$ and $e_2$ opposite sides in $\mathcal{P}_0$ as desired. Now assume $m$ is odd. Then $e_2$ is an edge incident the vertex opposite $e$ in $P_m$. We have $m+n-2$ odd, and $e_3$ is an (exterior) edge incident to vertex opposite $e_1$ in $\mathcal{P}_0$. 

Similar argument proves the induction step when $n$ is odd. (See Figure \ref{mediumbigon}).
\end{proof}
\begin{figure}[h]
\hspace{0cm}
\includegraphics[width=8cm, height=12cm]{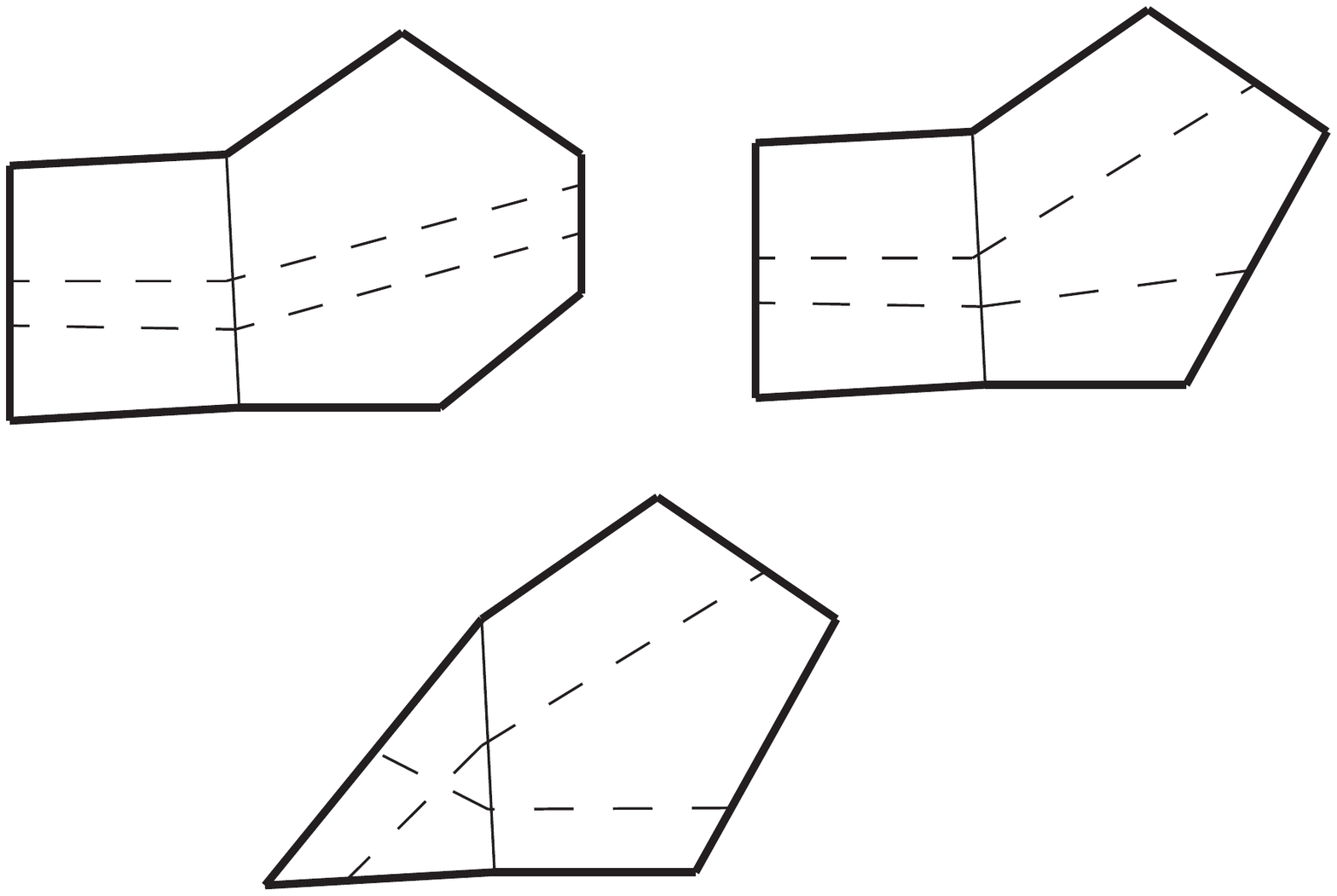}
\vspace{-6cm}
\caption{$\mathcal{P}_0$ union of polygons. Interior edge in thin line, exterior edges in bold, arcs of $\tilde\lambda_S$ that intersect interior edge shown in dashed lines.}
\label{mediumbigon}
\end{figure}

Next we briefly consider the case when $\Sigma$ has boundary components, and explain how the lemmas above also hold in this case. As above, let $\tilde\Sigma$ be its universal cover and $T$ the tiling corresponding to the simple generating set $S$. Since $\Sigma$ is not closed, some of the tiles in $T$ will be non-compact. Let $\tilde p$ be a path in $\tilde S$ that form a bigon with some arc $\tilde a$ of a leaf $\tilde\alpha\subset\tilde\lambda_S$. We claim that, with the notation as above, all tiles in $A_i$ for all $A_i$ whose interior $\tilde p$ or $\tilde a$ intersect, are in fact compact. To see this, not first that the claim follows trivially for $i>0$ since a bigon bounds a compact region. Hence we only need to show that the claim holds for $A_0=P_0$, and hence that every tile that $\tilde a$ intersects is compact. Suppose not, i.e. there exists a non-compact tile $\tilde P$ that $\tilde a$ intersects. Let $P$ be the corresponding one-holed polygon on $\Sigma\setminus S$ and $\alpha\subset\lambda_S$ the leaf corresponding to $\tilde \alpha$. But then, by the construction of $\lambda_S$, $\alpha$ has an endpoint on the boundary component, a contradiction to the fact $\tilde a$ forms a bigon with $\tilde p$. Hence the claim is true, and when $\vert S\vert\geq 3$ we can apply Lemmas \ref{Dehn} and \ref{path} also to non-compact surfaces. 

Lastly, we need to consider the case $\vert S\vert=2$. Then $\Sigma$ is a one-holed torus $\mathbb{T}$ and $S$ is a free basis for $\pi_1(\mathbb{T},p)$ and hence it differs from the standard generating set by an element of the mapping class group of $\mathbb{T}$. So $\mathbb{T}\setminus S$ is a one-holed square and the associated $\gamma_S$ consist of arcs between the hole and itself. It is clear, with similar argument as above, that no path in $\tilde S$ in the tiling of non-compact squares can form a bigon with such an arc. 

Armed with Lemmas \ref{Dehn} and \ref{path}, we can now prove Proposition \ref{bigonprop}, showing that a word is shortest (or cyclically shortest) if and only if the corresponding loop (or curve) does not make a bigon with $\lambda_S$. 

\begin{proof}[Proof of Proposition \ref{bigonprop}]
Suppose $w$ is shortest (or cyclically shortest) but forms a bigon with $\lambda_S$. Consider the lifts to the universal cover $\tilde\Sigma$, tiled by $T$. Each lift of $w$ is a path in $\tilde S$, choose one of them and call it $\tilde p$. This path forms a bigon with some arc $\tilde a$ of a leaf $\tilde\alpha\subset\tilde\lambda_S$. Let  $\mathcal{P}_0$ be the union of the polygons $\tilde a$ intersects. By Lemma \ref{Dehn}, since $w$ is shortest, $\tilde p$ must be contained in $\mathcal{P}_0$. By Lemma \ref{path}, if $\tilde p$ forms a bigon with $\tilde a$ it must traverse more than half the exterior edges of $\mathcal{P}_0$ and can hence be replaced with a shorter path, contradicting the fact that $w$ is shortest. Hence $w$ cannot form a bigon with $\lambda_S$. 

Now suppose that $w$, viewed as a loop, does not form any bigons with $\lambda_S$. Then $\vert w\cap\lambda_S\vert=\iota(w,\lambda_S)$. Suppose there exists a word $w'$ homotopic to $w$ that is shorter than $w$. Then it follows from Lemma \ref{atleast} that $\vert w'\cap\lambda_S\vert<\vert w\cap\lambda_S\vert=\iota(w,\lambda_S)=\iota(w',\lambda_S)$, a contradiction.

Similar argument proves the statement when viewing $w$ as a curve. 
\end{proof}


\section{Counting curves}\label{counting}

In this section we will prove Theorem \ref{count} but first we need to recall some facts and definitions about \emph{(geodesic) currents}. The reader is referred to \cite{Bon1, Bon2, Ota, AL} for more details. Endow the interior of $\Sigma$ with a complete hyperbolic metric, say of finite volume, to obtain the hyperbolic surface $\Sigma_1$. As introduced by Bonahon \cite{Bon1}, a (geodesic) current is a $\pi_1(\Sigma)$-invariant Radon measure on the space $\mathcal{G}(\tilde\Sigma_1)$ of geodesics in the universal cover of $\Sigma_1$. Let $\mathcal{C}(\Sigma)$ denote the set of all currents on $\Sigma$. $C(\Sigma)$ is independent of the choice of finite volume hyperbolic metric in the sense that if $\Sigma_2$ is another such structure on $\Sigma$, then the corresponding universal covers $\tilde{\Sigma}_1$ and $\tilde{\Sigma}_2$ are quasi-isometric and this quasi-isometry extends to a homeomorphism of their ideal boundaries, and hence of the spaces $\mathcal{G}(\tilde\Sigma_1)$ and $\mathcal{G}(\tilde\Sigma_2)$. 

We say a current $\lambda\in\mathcal{C}(\Sigma)$ is \emph{filling} if every geodesic in $\Sigma$ is transversally intersected by some geodesic in the support of $\lambda$. 

Any union of curves and arcs $\lambda\subset\Sigma$ defines a current in the following way: The set of lifts $\tilde\lambda$ of $\lambda$ to the universal cover $\tilde\Sigma$ is a $\pi_1(\Sigma)$-invariant discrete subset of $\mathcal{G}(\tilde\Sigma)$ and we can define a measure $\mu_{\lambda}$ which is supported on $\tilde\lambda$ by defining
$$\mu_{\lambda}(B)=\vert B\cap\tilde\lambda\vert$$
for any Borel set $B\subset\mathcal{G}(\tilde\Sigma)$. By abuse of notation we will identify $\lambda$ and the current it defines. This way the set of curves on $\Sigma$ can be viewed as a subset of $\mathcal{C}(\Sigma)$. In fact, the set of all \emph{weighted curves} $\{t \delta\,:\, t\in\mathbb{R}_+,\,\delta\mbox{ a curve}\}$ is dense in $\mathcal{C}(\Sigma)$. 

Recall that a \emph{measured lamination} on $\Sigma$ is a closed subset $\mathcal{L}\subset\Sigma$ of simple, pairwise disjoint geodesics together with a transverse measure supported on $\mathcal{L}$. Let $\ML(\Sigma)$ denote the space of measured laminations on $\Sigma$. (For the same reason as for currents, this makes sense without referencing to a metric on $\Sigma$). As above, we can consider the set of lifts $\tilde{\mathcal{L}}$ of $\mathcal{L}$ to the universal cover of $\Sigma$ and identify it with a measure on $\mathcal{G}(\tilde\Sigma)$ and this measure agrees with the transverse measure. In this way we can view $\ML(\Sigma)$ as a subset of $\mathcal{C}(\Sigma)$. 

Recall that if $\Sigma$ is of genus $g$ with $r$ boundary components, $\ML(\Sigma)$ is homeomorphic  to $\mathbb{R}^{6g-6+2r}$ and has thus a natural topology. The space $\ML(\Sigma)$ also has a mapping class group invariant integral PL-manifold structure (where integral means that the change of charts are given by linear transformations with integral coefficients). It is endowed with a mapping class group invariant measure in the Lebesgue class, the so-called {\em Thurston measure} $\mu_{\Thu}$ (see \cite{thurstonnotes}). It is an infinite but locally finite measure, positive on non-empty open sets, and satisfies
$$\mu_{\Thu}(L\cdot U)=L^{6g-6+2r}\mu_{\Thu}(U)$$ 
for all $U\subset\ML(\Sigma)$ and $L>0$. On charts, the measure $\mu_{\Thu}$ is just the standard Lebesgue measure.

In \cite{Bon1} Bonahon defines a symmetric map $\iota(\cdot,\cdot): \mathcal{C}(\Sigma)\times\mathcal{C}(\Sigma)\to\mathbb{R}_{+}$ which extends the geometric intersection form of curves on $\Sigma$ to the space of currents. This map, called the \emph{intersection form}, is symmetric, bi-homogenous, and invariant under the action of the mapping class group of $\Sigma$. When $\Sigma$ is closed, the intersection form is also continuous, but this fails in general. However, this problem can be solved if we work with convex cocompact surfaces (see \cite{DLR}) but in any case this issue has no relevance to this paper. 

Recall also that for each hyperbolic, or more generally any negatively curved, structure $X$ on $\Sigma$ there exists a filling current $\lambda_X$ associated to it, called the \emph{Liouville current} (see \cite{Bon1, Ota}) which satisfies 
\begin{equation}\label{liouville}
\iota(\gamma,\lambda_X)=\ell_X(\gamma)
\end{equation}
for all curves $\gamma\subset\Sigma$. In particular, one can view the current $\lambda_S$ associated to a simple generating set $S$ as a combinatorial version of the Liouville current. It is a theorem of Otal \cite[Theorem 2]{Ota} that a current is determined by its intersection with all curves, that is, two currents $\lambda, \mu\in\mathcal{C}(\Sigma)$ satisfy $\iota(\lambda,\gamma)=\iota(\mu, \gamma)$ for all curves $\gamma\subset\Sigma$ if and only if $\lambda=\mu$. Hence we obtain from Theorem \ref{currentfree} that the current $\lambda_S$ associated to a simple generating set is unique:

\begin{cor}\label{unique}
The collection of curves and arcs $\lambda_S$ satisfying Theorem \ref{currentfree} is unique up to homotopy.\qed
\end{cor} 


After these preliminary comments we return now to the setting of Theorem \ref{count}. More concretely, suppose $\Sigma$ has genus $g$ and $r$ boundary components, where $3g+r>3$. Let $\MCG(\Sigma)$ denote the mapping class group of $\Sigma$. Let $\gamma_0\subset\Sigma$ be a curve and consider the mapping class group orbit $\MCG(\Sigma)\cdot\gamma_0$ of $\gamma_0$. That is:
$$\MCG(\Sigma)\cdot\gamma_0 = \{\gamma\in\Sigma\,:\, \gamma=f(\gamma_0)\mbox{ for some } f\in\MCG(\Sigma)\}.$$

Let $X$ be a complete hyperbolic structure on $\Sigma$, and let $\ell_X(\gamma)$ denote the length with respect to $X$ of the unique geodesic in the homotopy class of $\gamma$. In \cite{Mir08} and \cite{MM2} Mirzakhani proved that for any curve $\gamma_0$, the number of (homotopy classes of) curves in the mapping class group orbit of $\gamma_0$ of length bounded by $L$ is asymptotic to a polynomial in $L$. More precisely she showed the following (stated here in the form we will use):

\begin{thm}[Mirzakhani, \cite{MM2} Theorem 1.1]\label{mir}
Let $X$ be a complete hyperbolic metric on $\Sigma$. Then for every immersed curve $\gamma_0$, 
 $$\lim_{L\to\infty}\frac{\vert\{\gamma\in\MCG(\Sigma)\cdot\gamma_0\,:\,\ell_X(\gamma)\leq L\}\vert}{L^{6g-6+2r}} =C_{\gamma_0}M_X$$
where $K_{\gamma_0}>0$ is a constant depending only on the mapping class group orbit of $\gamma_0$ and $M_X=\mu_{Thu}(\{\nu\in\ML(\Sigma)\,:\,\ell_X(\nu)\leq 1\})$.
\end{thm}

In \cite{ES} we investigated the existence of the more general limit 
\begin{equation}\label{limitcurrent}
\lim_{L\to\infty}\frac{\vert\{\gamma\in\MCG(\Sigma)\cdot\gamma_0\,:\, \iota(\gamma,\lambda)\leq L\}\vert}{L^{6g-6+2r}} 
\end{equation}
where $\lambda$ is any filling current on $\Sigma$. We showed that

\begin{thm}[Erlandsson-Souto, \cite{ES} Corollary 4.4]\label{es} 
Let $\lambda_1, \lambda_2\in\mathcal{C}(\Sigma)$ be two filling currents. Then
\begin{equation*}
\lim_{L\to\infty}\frac{\vert\{\gamma\in\MCG(\Sigma)\cdot\gamma_0\,:\,\iota(\gamma,\lambda_1)\leq L\}\vert}{\vert\{\gamma\in\MCG(\Sigma)\cdot\gamma_0\,:\,\iota(\gamma,\lambda_2)\leq L\}\vert}=\frac{\mu_{\Thu}(\{\nu\in\ML(\Sigma)\,:\,\iota(\nu,\lambda_1)\leq 1\})}{\mu_{\Thu}(\{\nu\in\ML(\Sigma)\,:\,\iota(\nu,\lambda_2)\leq 1\})}
\end{equation*}
for every immersed curve $\gamma_0\subset\Sigma$.
\end{thm}

We will use the two theorems above to prove Theorem \ref{count}.

\begin{proof}[Proof of Theorem \ref{count}]
Note that Theorem \ref{es} implies that the \emph{existence} of limit \eqref{limitcurrent} is independent of the choice of filling current $\lambda$. Let $X$ be a complete hyperbolic metric on $\Sigma$ and $\lambda_X$ the corresponding Liouville current. Then, substituting $\lambda_X$ for $\lambda$, limit \eqref{limitcurrent} exists by Mirzakhani (Theorem \ref{mir}). Hence limit \eqref{limitcurrent} exists for any filling current $\lambda\in\mathcal{C}(\Sigma)$. Moreover,   
\begin{equation}\label{finallimit}
\lim_{L\to\infty}\frac{\vert\{\gamma\in\MCG(\Sigma)\cdot\gamma_0\,:\,\iota(\gamma,\lambda)\leq L\}\vert}{L^{6g-6+2r}} = C_{\gamma_0}M_{\lambda}
\end{equation} 
for any filling current $\lambda$, where $M_{\lambda}=\mu_{Thu}(\{\nu\in\ML(\Sigma)\,:\,\iota(\nu,\lambda)\leq 1\})$. 

In particular, \eqref{finallimit} holds when replacing $\lambda$ with the current $\lambda_S$ associated to a simple generating set $S$ for $\pi_1(\Sigma,p)$. Hence, applying Theorem \ref{currentfree}, we have
$$\lim_{L\to\infty}\frac{\vert\{\gamma\in\MCG(\Sigma)\cdot\gamma_0\,:\,|\gamma|_S^c\leq L\}\vert}{L^{6g+2r-6}}=C_{\gamma_0}\cdot \mu_{Thu}(\{\nu\in\ML\,:\,|\nu|^c_{S}\leq1\}),$$
proving Theorem \ref{count}.
\end{proof}

Note that the proof of Theorem \ref{count}  gives no information on the constant $C_{\gamma_0}$. It is in fact a very interesting question to investigate how it depends on the curve $\gamma_0$. Recall that, as mentioned in the introduction, Chas \cite{Cha} has conjectured that the ratio $\frac{C_{\gamma_0}}{C_{\gamma_1}}$, for two curves $\gamma_0$ and $\gamma_1$, is always rational.   
\bibliographystyle{amsalpha}
\bibliography{ref}

\end{document}